\documentclass[12pt,a4paper,reqno]{amsart}
\usepackage[a4paper,top=3cm,bottom=3cm,inner=3cm,outer=3cm,includehead,includefoot]{geometry}
\usepackage{amsmath,amssymb,amsthm,wasysym,calc,verbatim,enumitem,tikz,url,hyperref,mathrsfs,cite}
\usetikzlibrary{shapes.misc,calc,intersections,patterns,decorations.pathreplacing}
\hypersetup{colorlinks=true,citecolor=blue,linktoc=page}

\newtheorem{theorem}{Theorem}
\newtheorem{lemma}[theorem]{Lemma}

\theoremstyle{definition}
\newtheorem{definition}[theorem]{Definition}

\setlist[itemize]{leftmargin=1cm}
\setlist[enumerate]{leftmargin=1cm}

\everydisplay{\thickmuskip=7mu}
\renewcommand{\leq}{\leqslant}
\renewcommand{\geq}{\geqslant}

\renewcommand{\to}{\rightarrow}

\let\eps\varepsilon

\def\E{\mathbb{E}}
\def\N{\mathbb{N}}
\def\P{\mathbb{P}}

\def\B{\mathcal{B}}
\def\C{\mathcal{C}}

\def\G{\mathcal{G}}

\def\S{\mathcal{S}}

\def\<{\langle}
\def\>{\rangle}
\def\bin{\mathrm{Bin}}

\def\er{Erd\H{o}s--R\'enyi}
\def\cc{\mathrm{c}}

\def\biggiven{\;\big|\;}
\def\Biggiven{\;\Big|\;}

\title{The threshold for jigsaw percolation on random graphs}

\author{B\'ela Bollob\'as \and Oliver Riordan \and Erik Slivken \and Paul Smith}

\address{Department of Pure Mathematics and Mathematical Statistics, Wilberforce Road, Cambridge CB3 0WA, UK, and Department of Mathematical Sciences, University of Memphis, Memphis, TN 38152, USA, and London Institute for Mathematical Sciences, 35a South Street, London W1K 2XF, UK}
\email{b.bollobas@dpmms.cam.ac.uk}
\address{Mathematical Institute, University of Oxford, Radcliffe Observatory Quarter, Woodstock Road, Oxford OX2 6GG, UK}
\email{riordan@maths.ox.ac.uk}
\address{Department of Mathematics, UC Davis, One Shields Avenue, Davis, CA 95616, USA}
\email{erikslivken@math.ucdavis.edu}
\address{Department of Pure Mathematics and Mathematical Statistics, Wilberforce Road, Cambridge CB3 0WA, UK}
\email{p.j.smith@dpmms.cam.ac.uk}

\thanks{B.B.\ is partially supported by NSF grant DMS~1301614 and MULTIPLEX grant no.~317532}

\date{\today}

\subjclass[2010]{Primary 05C80; Secondary 60C05}
\keywords{jigsaw percolation, random graphs}

\begin{document}

\begin{abstract}
Jigsaw percolation is a model for the process of solving puzzles within a social network, which was recently proposed by Brummitt, Chatterjee, Dey and Sivakoff. In the model there are two graphs on a single vertex set (the `people' graph and the `puzzle' graph), and vertices merge to form components if they are joined by an edge of each graph. These components then merge to form larger components if again there is an edge of each graph joining them, and so on. Percolation is said to occur if the process terminates with a single component containing every vertex. In this note we determine the threshold for percolation up to a constant factor, in the case where both graphs are \er~ random graphs.
\end{abstract}

\maketitle

\section{Introduction}

Jigsaw percolation is a dynamical percolation model on finite graphs, which was proposed by Brummitt, Chatterjee, Dey and Sivakoff~\cite{BCDS} as a tool for the study of sequences of interactions within a social network that enable a group of individuals to collectively solve a problem. In the model there are two edge sets defined on a common set of vertices, and, at discrete times, clusters of vertices merge to form larger clusters if they are joined by at least one edge of each type. Before expanding on the motivation for the model, let us give the formal definition. We write $[n]$ for $\{1,2,\ldots,n\}$.

\begin{definition}\label{def:jigsaw}
For $i=1,2$, let $E_i\subset[n]^{(2)}$ be a set of pairs of elements of $V:=[n]$. Let $G$ be the ordered triple $G:=(V,E_1,E_2)$; we call this object a \emph{double graph}. \emph{Jigsaw percolation} with input $G$ evolves at discrete times $t=0,1,\dots$ according to the following algorithm. At time $t$ there is a partition $\C_t=\{C_t^1,\dots,C_t^{k_t}\}$ of the vertex set $[n]$, which is constructed inductively as follows:
\begin{enumerate}
\item We take $k_0=n$ and $C_0^i = \{i\}$ for all $1\leq i\leq n$. That is, at time 0 we begin with every vertex in a separate set of the partition.
\item At time $t\geq 0$, construct a graph $\G_t$ on vertex set $\C_t$ by joining $C_t^i$ to $C_t^j$ if there exist edges $e_1\in E_1$ and $e_2\in E_2$ such that
\[
e_\ell \cap C_t^k \neq \emptyset
\]
for each of the four choices of $\ell\in\{1,2\}$ and $k\in\{i,j\}$.
\item If $E(\G_t)=\emptyset$, then STOP. Otherwise, construct the partition
\[
\C_{t+1}=\{C_{t+1}^1,\dots,C_{t+1}^{k_{t+1}}\}
\]
corresponding to the connected components of $\G_t$, so each part $C_{t+1}^i$ is a union of those parts of $\C_t$ corresponding to a component of $\G_t$.
\item If $|\C_{t+1}|=1$ then STOP. Otherwise, go to step 2.
\end{enumerate}
Since $|\C_t|$ is strictly decreasing, the algorithm terminates in time at most $\binom{n}{2}$. We denote the final partition by $\C_\infty=(C_\infty^1,\dots,C_\infty^{k_\infty})$. We say that there is \emph{percolation}, or that the double graph is \emph{solved}, if $\C_\infty=\{V\}$, i.e., if we stop in step (4).
\end{definition}

Less formally, the jigsaw percolation algorithm begins with each vertex considered a separate cluster, and proceeds by merging, at each step, clusters of vertices joined by at least one edge from $E_1$ and at least one edge from $E_2$.

Let us mention in passing a superficially similar, but very different, percolation model in double graphs introduced by Buldyrev, Parshani, Paul, Stanley and Havlin~\cite{BPRSH} in 2010. The set-up is the same, but one defines the partition in a top-down
way, finding the maximal sets of vertices connected in both graphs $(V,E_1)$ and $(V,E_2)$ (`mutually connected clusters'). To see the difference note that if these graphs are edge-disjoint and connected, then in this model $V$ forms a single `mutually connected cluster', whereas in jigsaw percolation the algorithm stops where it starts, with a partition into singletons.

Returning to the jigsaw model, which we consider throughout this paper, Brummitt, Chatterjee, Dey and Sivakoff~\cite{BCDS} suggest that jigsaw percolation may be a suitable model for analysing how a puzzle may be solved by collaboration between individuals in a social network. The premise is that each individual has a `piece' of the puzzle, and that these `pieces' must be combined in a certain way in order to solve the puzzle. The process of solving the puzzle is constrained by the social network of the individuals concerned. To model this, the authors of~\cite{BCDS} suggest that one of the graphs, $G_1:=(V,E_1)$ say, (which they call the \emph{people graph}, and which we call the \emph{red graph}), could represent the graph of acquaintances, and that the other graph, $G_2:=(V,E_2)$, (which they call the \emph{puzzle graph}, and which we call the \emph{blue graph}), could represent the `compatibility' between pairs of `pieces' of the puzzle. The jigsaw percolation algorithm thus represents the merging of `compatible puzzle pieces' by groups of connected individuals. For an in-depth account of the applications of the model to social networks, we refer the reader to the original article~\cite{BCDS}.

In~\cite{BCDS}, the authors prove a number of necessary and sufficient conditions for percolation when the red graph $G_1$ is the \er~ random graph $\G(n,p)$ with $n$ vertices and edge probability $p$ and the blue graph $G_2$ is a deterministic graph, such as the $n$-cycle or another graph of bounded maximum degree. For example, they show that there is an absolute constant $c>0$ such that if, for each $n\in\N$, $G_2^n=([n],E_2^n)$ is an arbitrary connected graph on vertex set $[n]$, and $G_1^n$ is an \er~ graph with edge probability $p\geq c/\log n$, then the double graph $G:=([n],E_1^n,E_2^n)$ percolates with high probability as $n\to\infty$. On the other hand, if the graphs $G_2^n$ have bounded maximum degree and instead $p<n^{-\eps}$ for some $\eps>0$, then with high probability the double graph does not percolate.

Gravner and Sivakoff~\cite{GS} observe that for certain deterministic graphs $G_2$, the jigsaw percolation model behaves similarly to bootstrap percolation on the grid $[n]^2$, and they use techniques from bootstrap percolation to prove tight bounds in certain cases. For example, if $G_1=\G(n,p)$ and $G_2=C_n$ is the $n$-cycle, they show that the \emph{critical probability} $p_\cc^{C_n}(n)$ for the corresponding double graph $G$, defined by
\[
p_\cc^{C_n}(n) := \inf \big\{ p \,:\, \P(G \text{ percolates}) \geq 1/2 \big\},
\]
satisfies
\[
p_\cc^{C_n} = \frac{(1+o(1))\pi^2/6}{\log n}.
\]
(This critical probability, and specifically the constant $\pi^2/6$, will be known to readers who are familiar with bootstrap percolation: it is also (see~\cite{Hol}) the critical probability for the so-called `modified' bootstrap percolation model on $[n]^2$ -- this is, of course, not a coincidence (see~\cite{GS} for the details).)

In this note we study the case where \emph{both} underlying graphs are \er~ random graphs.
In order to state our result, we need a little more notation. For the rest of the paper, we shall take $G_1=(V,E_1)$ and $G_2=(V,E_2)$ to be independent \er~ random graphs with the same vertex set $V=[n]$, with edge probabilities $p_1$ and $p_2$ respectively, and we take $G=([n],E_1,E_2)$.
A first trivial observation is that if the double graph $G$ is to percolate then both $G_1$ and $G_2$ must be connected. We shall ensure this by assuming
\begin{equation}\label{eq:conn}
\min\{ p_1,p_2 \} \geq \frac{c\log n}{n},
\end{equation}
for a sufficiently large absolute constant $c$.\footnote{With some tightening, our arguments could be made to work under the (optimal) assumption that $\min\{ p_1, p_2 \} \geq (1 + \eps)\log n / n$, where $\epsilon > 0$ is fixed but arbitrary. We choose to make the stronger assumption in~\eqref{eq:conn} for clarity of the exposition.} Under this condition, in this note we determine the critical value $p_\cc(n)$ of the product $p_1 p_2$ up to a constant factor. More precisely, we show that under the assumption~\eqref{eq:conn}, if $p_1 p_2\leq (1/c)p_\cc(n)$ then percolation is very unlikely, and if $p_1 p_2 \geq c p_\cc(n)$ then percolation is very likely.

\begin{theorem}\label{thm:pc}
There is an absolute constant $c>0$ such that the following holds, with $G=(V,E_1,E_2)$, where $(V,E_1)=\G(n,p_1)$ and $(V,E_2)=\G(n,p_2)$ are independent \er\ random graphs on the same set of $n$ vertices, and $p_1$ and $p_2$ are functions of $n$.
\begin{enumerate}
\item[(i)] If $p_1p_2\leq 1/(cn\log n)$, then $\P(G \text{ percolates})\to 0$.
\item[(ii)] If $p_1p_2\geq c/(n\log n)$ and \eqref{eq:conn} holds, then $\P(G \text{ percolates})\to 1$.
\end{enumerate}
\end{theorem}

Informally, this result says that
\[
p_\cc(n) = \Theta\left(\frac{1}{n\log n}\right).
\]

If $p_2$ satisfying~\eqref{eq:conn} is given, then another way of thinking about Theorem~\ref{thm:pc} is that the critical value of $p_1$ is $\Theta(1 / p_2 n \log n) = \Theta(1 / d \log n)$ (provided this also satisfies~\eqref{eq:conn}), where $d = (n - 1)p_2 $ is the expected degree of vertices in $G_2$. Theorems~1 and~2 of~\cite{GS} also show that, for certain families of (deterministic) $d$-regular graphs $G_2$ (with $d$ possibly being a function of $n$), the critical value of $p_1$ is $\Theta(1 / d \log n)$, although the authors of~\cite{GS} also show that this does not always hold.

The proof of lower bound in Theorem~\ref{thm:pc} is straightforward and follows from standard methods; the real content of this paper is the proof of the upper bound.

\section{Proof of part (i) of Theorem~\ref{thm:pc}}

Here we present the very brief proof of part (i) of Theorem~\ref{thm:pc}, although really it is no more than the argument used by Aizenman and Lebowitz~\cite{AL} to derive the lower bound up to a constant factor for the critical probability for 2-neighbour bootstrap percolation on $[n]^2$.

The percolation process defined above can be broken down into smaller steps, in each of which two clusters (parts of the current partition) merge -- specifically, one can modify step (3) of Definition~\ref{def:jigsaw} to merge an arbitrary pair of sets $C_t^i$ joined in the graph $\G_t$, rather than entire connected components. Since we start with a partition into singletons, considering the first step at which a cluster of size at least $\log n$ appears, it follows that if $(G,E_1,E_2)$ percolates, then there is some set $A$ of at least $\log n$ but at most $2\log n$ vertices such that the red and blue graphs restricted to $A$ are both connected. Using independence of the red and blue graphs, the facts that a connected graph must contain a spanning tree and that there are $k^{k-2}$ labelled trees on $k$ vertices, and the bound $\binom{n}{k}\leq (en/k)^k$, we see that
\begin{eqnarray*}
 \P(G \text{ percolates}) &\leq& \sum_{k=\log n}^{2\log n} \binom{n}{k} k^{k-2} p_1^{k-1} k^{k-2} p_2^{k-1} \\
 &\leq& \frac{1}{p_1p_2} \sum_{k=\log n}^{2\log n} \left( e n k p_1 p_2 \right)^k \\
 &\leq& \frac{1}{p_1p_2} \sum_{k=\log n}^{2\log n} \left( 2 e n p_1 p_2 \log n \right)^k \\
 &\leq& 2 e n \log n \sum_{k=\log n}^\infty \left( 2 e n p_1 p_2 \log n \right)^{k-1}.
\end{eqnarray*}
For $p_1p_2\leq  1/(e^4 n\log n)$, say, the quantity in brackets is at most $1/e^2$ and it follows that the final bound is $o(1)$, proving (i).

We can now move on to the main part of this paper: the second part of Theorem~\ref{thm:pc}.

\section{Proof of part (ii) of Theorem~\ref{thm:pc}}

Let us begin with a small number of conventions. As already mentioned, $G_1$ and $G_2$ will always be independent \er~ random graphs on vertex set $V:=[n]$, with densities $p_1$ and $p_2$ respectively. Throughout, we assume that $c$ is a sufficiently large absolute constant, that the number of vertices $n$ is sufficiently large, and that $p_1$ and $p_2$ satisfy
\begin{equation}\label{eq:conds}
p_1 p_2 = \frac{c}{n \log n} \qquad \text{and} \qquad \frac{c\log n}{n} \leq p_1 \leq p_2.
\end{equation}
The assumptions of Theorem~\ref{thm:pc}~(ii) require $p_1 p_2 \geq c/(n\log n)$ rather than the equality in \eqref{eq:conds}, but we may couple with smaller $p_1$ and $p_2$ if necessary so that~\eqref{eq:conds} holds. Constants implicit in $O(\cdot)$ notation (and its variants) are independent of $c$ (and of $n$).
For later use, let us note some immediate consequences of \eqref{eq:conds};
these follow since $p_1\leq \sqrt{p_1p_2}$ and $p_2=(p_1p_2)/p_1$:
\begin{equation}\label{eq:conds2}
 p_1 \leq \left(\frac{c}{n\log n}\right)^{1/2} \leq n^{-1/2} \qquad \text{and} \qquad p_2 \leq \frac{1}{(\log n)^2}.
\end{equation}

We need a key definition: that of an `internally spanned' set of vertices. The definition enables one to say which sets of sites (internally) percolate, without any help from other vertices, and is motivated by several similar notions in the bootstrap percolation literature (see, for example,~\cite{AL,CC}). Our method for showing that $G$ percolates will broadly take the form `there exists a nested sequence of internally spanned sets $U_1\subset\dots\subset U_m$, with $U_m = [n]$'; the crux will be finding such a sequence.

\begin{definition}
A set $U\subset[n]$ is \emph{internally spanned} by $G$ if the double graph $G^U:=(U,E_1^U,E_2^U)$ percolates, where $E_i^U$ is the edge set of the induced subgraph of $G_i$ on vertex set $U$, for $i=1,2$ (that is, $E_i^U:=E\big(G_i[U]\big)$). We write $I(G,m)$ for the event that $V(G)$ contains an internally spanned set of size at least $m$.
\end{definition}

The proof of the lower bound of Theorem~\ref{thm:pc} could now be rephrased as follows. First, observe that if $G$ percolates, then, by merging components two at a time, we have that $V(G)$ must contain an internally spanned set of size roughly $\log n$. Second, using well-known properties of trees, one can show that if $p$ is small then this event is unlikely to occur.

The proof of the upper bound is divided into three parts, with a corresponding division of both the red and blue edges into three subsets. In the first part of the proof we show that with high probability there is a set $A$ of at least $(\log n)^{3/2}$ vertices which is internally spanned by the first set of (red and blue) edges. This is the core of the proof: the `bottleneck' event to percolation (in a certain sense) is the existence of an internally spanned set of size about $\log n$.\footnote{This observation was also exploited in the proof of the lower bound, although it is nothing new: a similar idea was used in~\cite{BCDS,GS} on jigsaw percolation, and previously in~\cite{AL,Hol} (among many other papers) on bootstrap percolation.} Then we show, using the second set of edges, that with high probability the set $A$ is contained in an internally spanned (with respect to the edges reveal so far) set $B$ of size $n/16$. Finally, using the condition~\eqref{eq:conn}, we show using the third set of edges and the set $B$ that with high probability the whole vertex set is internally spanned.

In each of the first two parts we specify an `exploration algorithm', in which the edges of each of the underlying graphs are revealed in an order that depends on what has been observed so far. The purpose is to reveal as few edges as possible (in order that we may reveal them later if necessary) in the search for a nested sequence of internally spanned sets. The algorithms are set out explicitly in Definitions~\ref{def:alg1} and~\ref{def:alg2}.

Between the three parts of the proof, independence is maintained by \emph{sprinkling}: for each $i=1,2$ and $j=1,2,3$, we take $G_i^{(j)}$ to be an independent copy of $\G(n,p_i)$, where $p_1$ and $p_2$ satisfy the conditions \eqref{eq:conds} as before; we then set $E_i^{(j)}:=E(G_i^{(j)})$,
\[
G_i := G_i^{(1)} \cup G_i^{(2)} \cup G_i^{(3)}, \qquad \text{and} \qquad G^{(j)} := \big([n],E_1^{(j)},E_2^{(j)}\big).
\]
Constructing the $G_i$ in this way maintains both conditions in \eqref{eq:conds}, with a different value of $c$. More precisely, the edge probability $p_i'$ of $G_i$ satsfies $1-p_i' = (1-p_i)^3$, so $p_i'$ is a little less than $3p_i$. One could therefore replace $c$ by $9c$ in~\eqref{eq:conds}, and any double graph satisfying the new conditions could be coupled with our double graph.

\subsection{Part I}

In this first part of the proof we prove the following lemma.

\begin{lemma}\label{lem:part1}
The probability that $G^{(1)}$ contains an internally spanned set of size at least $(\log n)^{3/2}$ is at least
$1-e^{-\sqrt{n}}$.
\end{lemma}

We prove the lemma by repeatedly attempting to build an internally spanned set of size at least $(\log n)^{3/2}$ by adding one vertex at a time to a so-called `trial set' (we call this the \emph{1-by-1 algorithm}). If we find a suitable vertex to add to the trial set, then we continue. If not, then we discard the vertices from the trial set, and start again; we call this starting a new \emph{round}. Discarding the trial set will ensure independence between rounds (see below). More precisely, the algorithm performs a sequence of `tests', asking whether certain potential red edges or potential blue edges are present. (More precisely still, the `test' corresponding to a pair $\{x,y\}$ of vertices and $i\in \{1,2\}$ asks whether $xy\in E_i^{(1)}$.) We shall make sure that no test is performed twice.

The subtlety is in the order in which we reveal the edges: the aim is to reveal as few edges as possible in the search for each new vertex. Given an internally spanned trial set $X$, we first reveal all red edges from (not-yet-discarded) vertices outside $X$ to the \emph{most recently added} vertex $v$ in $X$. Since a potential red edge is tested immediately after the first time one of its ends is added to $X$, it cannot be tested twice within a round.\footnote{Looking at it from the point of view of vertices, rather than edges, the fact that a potential new vertex has been considered at the $t$th step, and found not to have a red edge to the most recently added vertex, does not stop us from testing the same vertex again at later steps, since in those steps we will be testing for different red edges.} Let $R$ be the set of vertices outside $X$ incident with such red edges. We test for blue edges, to the whole of the trial set, \emph{only} from vertices in $R$. If there is a vertex in $R$ with a blue edge to any vertex in the trial set, then we add one such vertex to the trial set. We discard all other vertices in $R$ until the end of this round; this ensures that no potential blue edge is tested twice within a round. At the end of a round we permanently discard all vertices in the trial set. Since a tested edge (red or blue) always has at least one end in the trial set, this ensures that no edge is tested in two different rounds, giving us the independence we need.

Here is a formal description of the algorithm.

\begin{definition}\label{def:alg1} {\bf (The 1-by-1 algorithm.)}
The algorithm is divided into \emph{rounds}, indexed by $k$, and each round is divided into \emph{steps}, indexed by $t$. At the start of the $k$th round there is a set $A_k\subset [n]$ of active vertices and a set $D_k\subset [n]$ of discarded vertices. We begin with $A_1=[n]$ and $D_1=\emptyset$. The procedure for the $k$th round is as follows:
\begin{enumerate}
\item At the start of the $t$th step of the $k$th round there is a set $X_k^t=\{x_k^1,\dots,x_k^t\}\subset A_k$ of trial vertices, a set $A_k^t\subset A_k$ of active vertices, and a set $D_k^t\subset A_k$ of discarded vertices. These sets partition $A_k$, so for all $t$, $A_k$ is the disjoint union of $X_k^t$, $A_k^t$ and $D_k^t$. To begin, we have $X_k^0=D_k^0=\emptyset$ and $A_k^0=A_k$.
\item For $t=0$, move an arbitrary active vertex to the trial set. That is, set $X_k^1=\{x_k^1\}$, $D_k^1=\emptyset$ and $A_k^1=A_k^0\setminus\{x_k^1\}$, where $x_k^1\in A_k^0$ is arbitrary.
\item For $t\geq 1$, reveal all edges of $G_1^{(1)}$ (that is, all red edges from the first sprinkling) between $A_k^t$ and $\{x_k^t\}$, and let
\[
R_k^t := \big\{ x\in A_k^t \,:\, xx_k^t\in E_1^{(1)} \big\}.
\]
Then, reveal all edges of $G_2^{(1)}$ (that is, all blue edges from the first sprinkling) between $R_k^t$ and $X_k^t$, and let 
\[
B_k^t := \big\{ x\in R_k^t \,:\, xx_k^s\in E_2^{(1)} \text{ for some } 1\leq s\leq t \big\}.
\]
\item If $B_k^t\neq\emptyset$, then let $x_k^{t+1}$ be an arbitrary element of $B_k^t$. Then set
\[
X_k^{t+1} := X_k^t \cup \{x_k^{t+1}\}, \quad A_k^{t+1} := A_k^t \setminus R_k^t, \quad \text{and} \quad D_k^{t+1} := D_k^t \cup R_k^t \setminus \{x_k^{t+1}\}.
\]
If $t\geq (\log n)^{3/2}$ then STOP, otherwise set $t:=t+1$ and go to step (3).
\item If $B_k^t=\emptyset$, then set
\[
A_{k+1} := A_k \setminus X_k \qquad \text{and} \qquad D_{k+1} := D_k \cup X_k.
\]
If
\[
k \geq \frac{n}{2(\log n)^{3/2}}
\]
then STOP, otherwise set $k:=k+1$ and $t:=0$, and go to step (1).
\end{enumerate}
\end{definition}

Before starting the analysis, let us note that since we consider at most $n/(2(\log n)^{3/2})$ rounds, and stop each with a trial set of size at most $(\log n)^{3/2}$, we start each round with
\begin{equation}\label{eq:Ak0}
|A_k^0| = |A_k| \geq n/2.
\end{equation}
Let $\B_k^t$ be the event that $X_k^{t+1}$ is (defined and) has size $t+1$ (if $t \geq 1$ then this is equivalent to the event that $B_k^t$ is (defined and) non-empty). We shall show that $\B_k^t$ is not too unlikely. For technical reasons, we also need to consider the event
\[
\S_k^t = \big\{ |R_k^s|\leq n^{3/4} \text{ for } s=1,2,\dots,t \big\}
\]
that within round $k$, we have not `used up' too many vertices by step $t$. (Here and in what follows we ignore rounding to integers in expressions such as $n^{3/4}$. This makes essentially no difference.) For $k\leq n/(2(\log n)^{3/2})$ and $t\geq 1$, let
\[
r_k^t:=\P\big( \B_k^t\cap \S_k^t \biggiven \B_k^{t-1} \cap \S_k^{t-1} \big),
\]
noting that $\S_k^0$ is the trivial event that always holds

We shall need two different estimates on $r_k^t$, according to whether $t$ is larger or smaller than $(\log n)/c$.

\begin{lemma}\label{lem:rkt}
Suppose that $k\leq n / \big( 2(\log n)^{3/2} \big)$ and $1\leq t\leq t_1=(\log n)^{3/2}$. Then
\[
r_k^t \geq \begin{cases}
1-\exp\big( - (n/5) p_1 p_2 t (1-p_2 t) \big) & \text{unconditionally,} \\
(1/10) n p_1 p_2 t (1-p_2 t) & \text{if } n p_1 p_2 t \leq 1.
\end{cases}
\]
\end{lemma}

\begin{proof}
We condition on the outcome of the exploration so far, up to the start of step $t$ of round $k$. Note that this information determines whether $\B_k^{t-1}\cap \S_k^{t-1}$ holds; we may assume that it does.

Given the information revealed so far, the conditional distribution of $|R_k^t|$ is binomial $\bin(|A_k^t|,p_1)$. Since $|A_k^t|\leq n$, we have
\begin{multline*}
\P\big( (\S_k^t)^\cc \biggiven \B_k^{t-1}\cap \S_k^{t-1} \big) = \P\big( |R_k^t| > n^{3/4} \biggiven \B_k^{t-1}\cap \S_k^{t-1} \big) \\
\leq \binom{n}{n^{3/4}} p_1^{n^{3/4}} \leq \big(e n^{1/4} p_1)^{n^{3/4}} \leq e^{-\sqrt{n}},
\end{multline*}
say, where for the first inequality we have taken a union bound and for the final inequality we have used the bound $p_1\leq 1/\sqrt{n}$ (from \eqref{eq:conds2}).

Now, conditional on the exploration so far, for each $x\in A_k^t$ we have
\[
\P\big( x\in B_k^t \big) = p_1 (1-(1-p_2)^t) \geq p_1 (1-e^{-p_2 t}) \geq p_1 p_2 t (1-p_2 t),
\]
since $1 - e^{-x} \geq x(1 - x)$ for all $x \geq 0$. On the event $\B_k^{t-1}\cap \S_k^{t-1}$ we have
\[
|A_k^t| \geq |A_k^0| - tn^{3/4} \geq n/4,
\]
using \eqref{eq:Ak0}. Since $\B_k^t$ holds if and only if $B_k^t\ne\emptyset$, we thus have
\begin{align*}
\P\big( (\B_k^t)^\cc \;\big|\; \B_k^{t-1}\cap \S_k^{t-1} \big)
&\leq \big( 1-p_1 p_2 t (1-p_2 t) \big)^{n/4} \\
&\leq \exp\big( -(n/4) p_1 p_2 t (1-p_2 t) \big),
\end{align*}
so
\[ 
r_k^t \geq 1-  \exp\big( -(n/4) p_1 p_2 t (1-p_2 t) \big) - e^{-\sqrt{n}}.
\]
The first case of the lemma now follows from the bounds
\[
n p_1 p_2 t = \frac{ct}{\log n} = o\big(\sqrt{n}\big) \qquad \text{and} \qquad p_2t\leq \frac{(\log n)^{3/2}}{(\log n)^2} = o(1).
\]
The second case follows from the first and the inequality $1-e^{-x}\geq x/2$, valid for $x\leq1$.
\end{proof}

In the next two lemmas, we break down the 1-by-1 algorithm into two stages: first, in Lemma~\ref{lem:trial1}, we show that the probability the algorithm reaches step
\[
t_0:=(\log n)/c
\]
in a given round is at least $n^{-O(1)/c}$. Then, in Lemma~\ref{lem:trial2}, we show that the probability it reaches step
\[
t_1:=(\log n)^{3/2},
\]
given that it has reached step $t_0$, is also at least $n^{-O(1)/c}$.

\begin{lemma}\label{lem:trial1}
Suppose that $k\leq n / \big( 2(\log n)^{3/2} \big)$. Then
\[
\P\big( \B_k^{t_0}\cap \S_k^{t_0} \biggiven \B_k^0 \big) \geq n^{-4/c}.
\]
\end{lemma}

\begin{proof}
The definition of $t_0$ combined with the expression for $p_1 p_2$ from~\eqref{eq:conds} implies that $np_1 p_2 t_0 \leq 1$, and hence the second case of Lemma~\ref{lem:rkt} applies for the whole range. Thus,
\[
\P\big( \B_k^{t_0}\cap \S_k^{t_0} \biggiven \B_k^0 \big) \geq \prod_{t=1}^{t_0} \frac{1}{10} n p_1 p_2 t (1-p_2 t) \geq \left(\frac{c}{10\log n}\right)^{t_0} t_0! (1-p_2 t_0)^{t_0}.
\]
Recall that $p_2\leq (\log n)^{-2}$, so $p_2t_0\leq (\log n)^{-1}=o(1)$. Noting that $1-x\geq e^{-2x}$ if $x\leq 1/2$ and that $t!\geq (t/e)^t$ for all $t\in\N$, it follows that
\[
\P\big( \B_k^{t_0}\cap \S_k^{t_0} \biggiven \B_k^0 \big) \geq \left(\frac{ct_0}{10e\log n}\right)^{t_0} \exp\big( -2p_2 t_0^2 \big) =\exp\big( -t_0\log(10e) - 2p_2 t_0^2 \big).
\]
Since $p_2 t_0 \leq 1/\log n$, we obtain
\[
\P\big( \B_k^{t_0}\cap \S_k^{t_0} \biggiven \B_k^0 \big) \geq \exp( -4t_0 ) = n^{-4/c},
\]
as required.
\end{proof}

\begin{lemma}\label{lem:trial2}
Suppose that $k\leq n / \big( 2(\log n)^{3/2} \big)$. Then
\[
\P\big( \B_k^{t_1} \biggiven \B_k^{t_0}\cap \S_k^{t_0} \big) \geq n^{-O(1)/c}.
\]
\end{lemma}

\begin{proof}
For $t_0<t\leq t_1$, we use the first of the two estimates in Lemma~\ref{lem:rkt}, the validity of which does not depend on $t$. Using this, and recalling that $p_2 \leq 1/(\log n)^2$ (from \eqref{eq:conds2}) and $t\leq (\log n)^{3/2}$, we have
\begin{align*}
\P\big( \B_k^{t_1} \biggiven \B_k^{t_0}\cap \S_k^{t_0} \big)
&\geq \prod_{t=t_0+1}^{t_1} \Big( 1-\exp\big( - (n/5) p_1 p_2 t (1-p_2 t) \big) \Big) \\
&\geq \prod_{t=t_0}^{t_1} \Big( 1-\exp\big( - (n/6) p_1 p_2 t \big) \Big) \\
&\geq \prod_{t=t_0}^{t_1} \Big( 1-\exp\big( - c t/(6\log n) \big) \Big) \\
&\geq \exp \bigg( - 3 \sum_{t=t_0}^{t_1} \exp\big( - c t / (6 \log n) \big) \bigg),
\end{align*}
where for the final step we used the inequality $1-x\geq \exp(-3x)$, valid (by convexity) for $0\leq x\leq 0.9$, say.\footnote{What we really need in this argument is that $p_2 \ll 1/\log n$, or equivalently $p_1 \gg 1/n$, which already implies the existence of a giant component in each colour. We only need the connectivity condition once we have obtained an internally spanned set of linear size.} Thus,
\[
 \P\big( \B_k^{t_1}  \biggiven \B_k^{t_0}\cap \S_k^{t_0} \big) \geq \exp\left( - \frac{3e^{-1/6}}{1-e^{-c/(6\log n)}} \right) \geq n^{-O(1)/c},
\]
where the implied constant does not depend on $c$.
\end{proof}

We now put the previous few lemmas together.

\begin{proof}[Proof of Lemma~\ref{lem:part1}]
Let $k\leq n/\big(2(\log n)^{3/2}\big)$. Then in the $k$th round, the probability of finding an internally spanned set of size $(\log n)^{3/2}$ is at least $n^{-O(1)/c}$, by applying Lemmas~\ref{lem:trial1} and~\ref{lem:trial2} in turn. Moreover, this bound holds conditional on the result of all previous bounds, since in proving Lemma~\ref{lem:trial1} and~\ref{lem:trial2} we conditioned on these previous rounds. Hence the probability that all $n/\big(2(\log n)^{3/2}\big)$ rounds terminate `early' (without finding an internally
spanned set of size $(\log n)^{3/2}$) is at most
\[
\Big( 1-n^{-O(1)/c} \Big)^{n/(2(\log n)^{3/2})} \leq \exp\left( -\Omega\left(\frac{n^{1-O(1)/c}}{(\log n)^{3/2}}\right)\right) \leq \exp\big( -\sqrt{n} \big),
\]
if $c$ is sufficiently large.
\end{proof}

\subsection{Part II}

In this part of the proof we prove the following lemma.

\begin{lemma}\label{lem:part2}
Given that $G^{(1)}$ contains an internally spanned set of size at least $(\log n)^{3/2}$, the conditional probability that $G^{(1)}\cup G^{(2)}$ contains an internally spanned set of size at least $n/16$ is at least $1-n^{-100}$. That is,
\[
\P \Big( I\big( G^{(1)}\cup G^{(2)} , n/16 \big) \Biggiven I\big( G^{(1)}, (\log n)^{3/2} \big) \Big) \geq 1-n^{-100}.
\]
\end{lemma}

In this range we use a different vertex exploration algorithm in order to find successively larger internally spanned sets. Rather than adding vertices 1-by-1, as in Part I, we attempt to double the size of the trial set at each step. We start with a set $X_0$ of size $t_1=(\log n)^{3/2}$ internally spanned by $G^{(1)}$ (found in Part I), and we continue so that at step $t$ we have a set $X_t$ of size
\[
 x_t := 2^t t_1
\]
internally spanned by $G^{(1)}\cup G^{(2)}$. In order to maintain independence between steps, we only add a new vertex $v$ to the trial set if there is at least one edge of each colour from $G^{(2)}$ joining $v$ to the subset of the trial vertices that was added at the previous step. 

\begin{definition}\label{def:alg2} {\bf (The doubling algorithm.)}
At the start of the $t$th step there is a set $X_t$ of vertices internally spanned by $G^{(1)}\cup G^{(2)}$, where $|X_t| = x_t$. The set $X_t$ is the \emph{trial set}. The set $A_t:=V(G)\setminus X_t$ of remaining vertices in the graph is the \emph{active set}. The algorithm takes as its inputs the double graphs $G^{(1)}$ and $G^{(2)}$, and a set $X_0$ of size $(\log n)^{3/2}$, internally spanned by $G^{(1)}$.
\begin{enumerate}
\item At step $t\geq 0$, reveal all edges of $G^{(2)}$ between $A_t$ and $X_t\setminus X_{t-1}$, where we set
$X_{-1}=\emptyset$. Let
\[
B_t := \bigcup_{v',v''\in X_t\setminus X_{t-1}} \Big\{ v\in A_t \,:\, vv'\in E_1^{(2)} \text{ and } vv'' \in E_2^{(2)} \Big\}.
\]
Thus, $B_t$ is the set of active vertices joined to $X_t\setminus X_{t-1}$ by an edge of each colour from the second sprinkling.
\item If $|B_t|\leq x_t$ then STOP. Otherwise, let $C_t\subset B_t$ be an arbitrary set of exactly $x_t$ vertices of $B_t$, and set
\[
X_{t+1} := X_t \cup C_t, \qquad \text{and} \qquad A_{t+1} := A_t \setminus C_t.
\]
If $|X_{t+1}|\geq n/16$ then STOP, otherwise go to step (1).
\end{enumerate}
\end{definition}

First we need a lower bound on the probability that the size of the trial set doubles at step $t$.

\begin{lemma}\label{lem:double}
The probability that $|B_t|\geq x_t$, conditional on the doubling algorithm having reached the $t$th step (that is, $|X_t| = x_t$), is at least
\[
1- \exp\Big( -\Omega\big( c(\log n)^2 \big) \Big).
\]
\end{lemma}

\begin{proof}
Let $s_t$ be the probability in question, so
\[
s_t = \P\big( |B_t|\geq x_t \biggiven |X_t| = x_t \big).
\]
Let $q_{t,i}$ be the probability that a vertex $v\in A_t$ is joined to $X_t \setminus X_{t-1} = C_{t-1}$ by at least one edge of $G_i^{(2)}$. Since $|C_{t-1}|=x_t/2$ for $t\geq 1$ while $|C_{-1}|=x_0$, we have
\begin{equation}\label{eq:qit}
q_{t,i} \geq 1-(1-p_i)^{x_t / 2} \geq 1-e^{-p_i x_t / 2} \geq \begin{cases} p_i x_t / 4 & \text{if } p_ix_t<2, \\ 1/2 & \text{otherwise.} \end{cases}
\end{equation}
By the definition of the doubling algorithm, we have $|X_t|\leq n/16$ (otherwise we would have stopped), so there are at least $15n/16\geq n/2$ vertices $v\in A_t$. The events that individual vertices are in $B_t$ are independent (because we do not `re-test' edges). Hence if $Z$ is a random variable with the $\bin\big(n/2,q_{t,1} q_{t,2}\big)$ distribution, then $s_t \geq \P( Z \geq x_t )$.

From \eqref{eq:qit} it is easy to check that
\[
\E[Z]= (n/2) q_{t,1} q_{t,2} \geq 2 x_t.
\]
Indeed, if $p_1x_t\leq p_2x_t<2$, then $\E[Z] \geq (n/32)p_1p_2x_t^2 \geq 2x_t$ since $np_1p_2\geq c/(\log n)$ (from \eqref{eq:conds}) and $x_t\geq (\log n)^{3/2}$. If $p_1x_t<2\leq p_2x_t$, then $\E[Z]\geq np_1x_t/16\geq 2x_t$, using $np_1\geq c\log n$. Finally, if $2<p_1x_t\leq p_2x_t$, then $\E[Z]\geq n/8\geq 2x_t$ since $x_t\leq n/16$. We thus have
\[
 s_t \geq \P(Z\geq x_t) \geq \P\big(Z\geq \E[Z]/2\big) \geq 1-\exp\big(-\Omega(\E[Z])\big),
\]
using a standard (Chernoff-type) bound for the final step. Now $q_{t,i}$ is increasing in $t$, so
\[
 \E[Z] = nq_{t,1}q_{t,2}/2 \geq n q_{0,1}q_{0,2}/2 \geq c(\log n)^2,
\]
where the last inequality follows from \eqref{eq:qit}, recalling that $x_0=t_1=(\log n)^{3/2}$
and that $p_1\leq p_2\leq (\log n)^{-2}$ (see \eqref{eq:conds2}).
This completes the proof.
\end{proof}

Lemma~\ref{lem:part2} is now immediate.

\begin{proof}[Proof of Lemma~\ref{lem:part2}]
Recall that we wish to show that
\begin{equation}\label{eq:part2}
\P \Big( I\big(  G^{(1)}\cup G^{(2)} , n/16 \big) \Biggiven I\big( G^{(1)} , (\log n)^{3/2} \big) \Big) \geq 1-n^{-100}.
\end{equation}
Let $t_2$ be maximal such that $x_{t_2} \leq n/16$, which in particular implies that $t_2=O(\log n)$. By Lemma~\ref{lem:double}, the left-hand side of~\eqref{eq:part2} is at least
\begin{align*}
\prod_{t=0}^{t_2} \P\big( |B_t|\geq x_t \biggiven |X_t| = x_t \big)
&\geq \Big(1-\exp\big(- \Omega(c(\log n)^2) \big) \Big)^{O(\log n)} \\
&\geq 1-\exp\big(- \Omega(c(\log n)^2) \big),
\end{align*}
which is certainly at least $1-n^{-100}$, as required.
\end{proof}

\subsection{Part III}

It remains to show that if $G^{(1)}\cup G^{(2)}$ contains an internally spanned set $X$ of size at least $n/16$ then $G=G^{(1)}\cup G^{(2)}\cup G^{(3)}$ is internally spanned with high probability. But this is trivial: using the final sprinkle, i.e., the edges of $G^{(3)}$, every vertex $v\in[n]\setminus X$ is joined to $X$ by both a red edge and a blue edge, with high probability.

\begin{proof}[Proof of Theorem~\ref{thm:pc}]
As noted above, it remains only to prove (ii), and in doing so, we may assume \eqref{eq:conds}. By Lemmas~\ref{lem:part1} (Part I) and~\ref{lem:part2} (Part II), $G^{(1)}\cup G^{(2)}$ contains an internally spanned set of size at least $n/16$ with probability at least $1-2n^{-100}$. Let $X$ be such a set. Then the probability that there is any vertex of $[n]\setminus X$ not joined to $X$ by at least one edge of $G_1^{(3)}$ (a red edge from the third sprinkling) and at least one edge of $G_2^{(3)}$ (a blue edge from the third sprinkling) is at most
\[
2n(1-p_1)^{n/16} = o(1),
\]
since $p_1 \geq c(\log n)/ n$ and $c$ is sufficiently large.
\end{proof}

\bibliographystyle{amsplain}
\bibliography{jigsaw,bprefs}

\end{document}